\documentclass[a4paper,12pt,reqno]{amsart}
\usepackage{amsmath}
\usepackage{amsfonts}
\usepackage{amssymb}
\usepackage{graphicx}
\usepackage[all]{xy}
\usepackage{a4wide}

\newtheorem{theorem}{Theorem}
\theoremstyle{plain}
\newtheorem{acknowledgement}[theorem]{Acknowledgement}

\newtheorem{lemma}[theorem]{Lemma}

\begin{document}
\title{an elementary proof of borsuk theorem}
\author{Dian Yang}
\date{\today{}}

\maketitle{}

In 1933, Borsuk conjectured that any bounded $d$-dimensional set of nonzero diameter can be broken into $d + 1$ parts of smaller diameter\cite{KB}. This conjecture was disproved for large enough $d$\cite{JG,N,R1,R2,W,HR}, though it is true for low dimensional cases. The paper provides an alternative proof for $d=2$ case.

\begin{theorem}[Borsuk]
Any bounded plane figure can be divided into three pieces of smaller diameters.
\end{theorem}

This theorem has a standard proof by first proving the plane figure (assume its diameter to be 1) can be bounded by a hexagon, with opposite sides parallel and separated by a distance no greater than 1, all of its angles being  $120^{\circ}$,  arbitrarily chosen the direction of one of its sides. By considerations of continuity we may assume that the hexagon is regular. It can be easily partitioned into three congruent pentagons with diameter $\sqrt{3}/2$. Instead of considerations of continuity, one can use the following lemma:
\begin{lemma}[Main Lemma]
A hexagon with opposite sides parallel and separated by distance 1, all of its angles being  $120^{\circ}$, can be divided into 3 parts of diameter less than 1.
\end{lemma}

\begin{proof}
(Hexagon $ABCDEF$ as shown in Figure 1.)

By construction, $|AB|=|CD|=|EF|$, $|BC|=|DE|=|FA|$. Without loss of generality we assume $|AB|\leq|BC|$. Drop perpendiculars $OP$, $OR$. Hexagon $ABCDEF$ is partitioned into three pentagons congruent to $ABPOR$. It suffices to prove that this pentagon has diameter smaller than 1, or equivalently: $|AP|<1$, $|RP|<1$, $|OB|<1$. 
Extend $BC$ and $FA$ so that they intersect at $G$. Since distances between opposite sides of $ABCDEF$ are 1's, $|AB|+|BC|=|GC|=2/\sqrt{3}$. Hence $|AB|\leq 1/\sqrt{3}$. 

On one hand, since $\Delta GPR$ is equilateral,
\begin{equation}
|AP|\le|RP|=|PG|<|OG|;
\end{equation}
and since $OP\bot GP$, 
\begin{equation}
|OB|\le|OG|.\quad \textnormal{($|OB|=|OG|$ iff. $|AB|=0$)}
\end{equation}
On the other hand,
\begin{equation}
|OG|=\frac{2}{\sqrt{3}}|PG|=\frac{1}{\sqrt{3}}|HG|=\frac{1}{\sqrt{3}}\left(|AB|+\frac{2}{\sqrt{3}}\right)\le1
\end{equation}
We conclude: $|AP|<1$, $|RP|<1$, $|OB|<1$, thus the lemma is proven.

\end{proof}

\begin{figure}
  \begin{center}
    \includegraphics{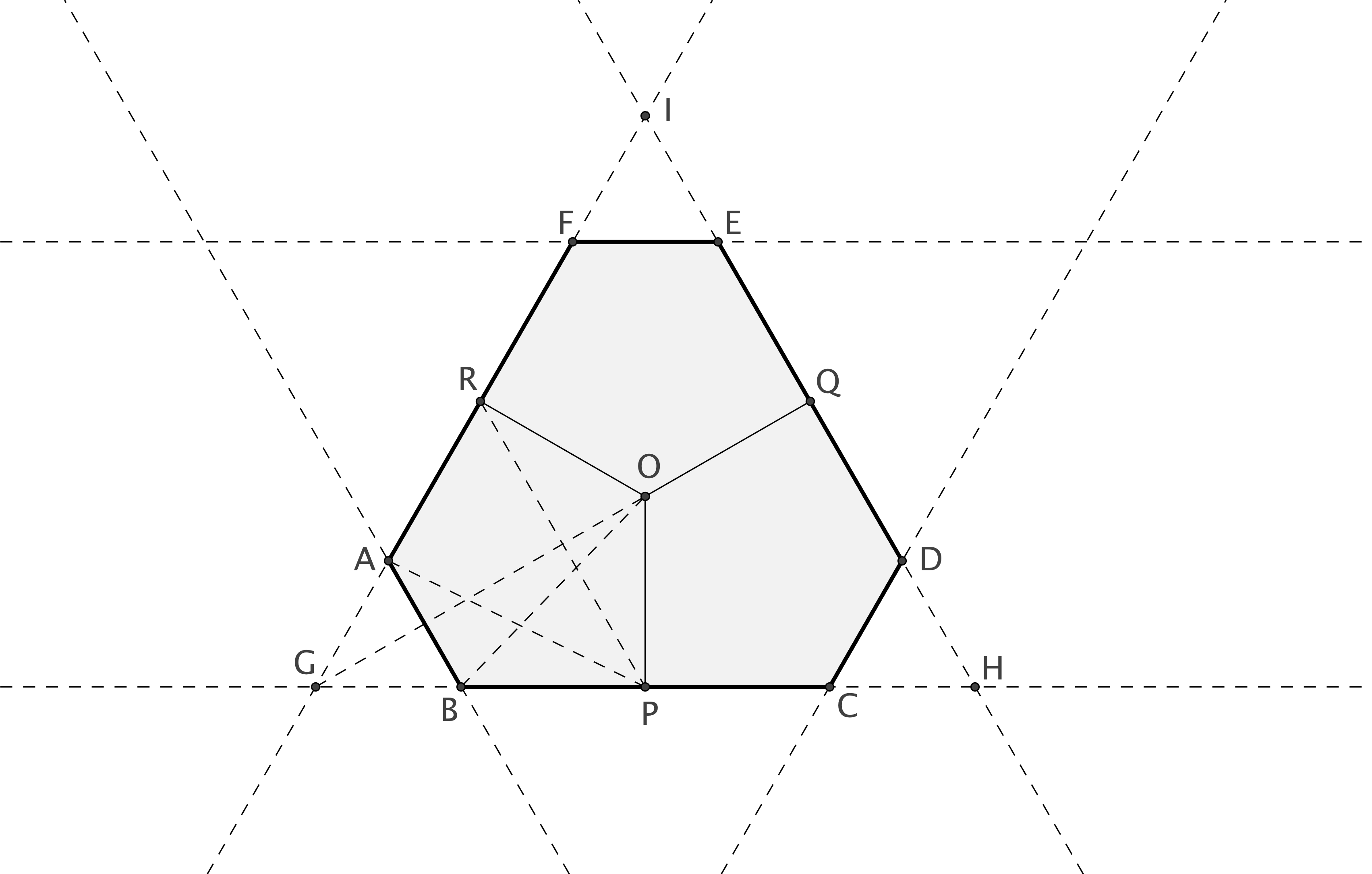}
  \caption{A hexagon with opposite sides parallel and separated by distance 1, all of its angles being  $120^{\circ}$}
  {\it The plane figure $\Phi$ (not shown) is covered by the hexagon.}
  \end{center}
\end{figure}

\acknowledgement{This note is one of the results of my participation in `Math in Moscow' program. I would like to thank the program as well as prof.\! A.\! Skopenkov of Moscow State University for his support and help in making the note more succinct.}

\end{document}